\documentclass{amsproc}

\usepackage{amssymb}

\usepackage{graphicx}

\usepackage{enumitem}  

\usepackage{fullpage} 
\usepackage{parskip} 
\usepackage{tikz} 
\usepackage{amsmath,amssymb,amsfonts,amsthm}
\usepackage{hyperref}
\usepackage{longtable}
\usepackage{indentfirst}
\usepackage{xcolor}
\usepackage{mathrsfs}
\usepackage{float}
\DeclareMathOperator{\rank}{rank}
\DeclareMathOperator{\Rep}{Rep}

\DeclareMathOperator{\Sem}{Sem}

\DeclareMathOperator{\Id}{Id}

\DeclareMathOperator{\PSU}{PSU}

\DeclareMathOperator{\sVec}{sVec}

\DeclareMathOperator{\Hom}{Hom}

\DeclareMathOperator{\Irr}{Irr}
\DeclareMathOperator{\1}{\textbf{1}}

\DeclareMathOperator{\diag}{diag}

\newcommand{\CC}{\mathcal{C}}
\newcommand{\DD}{\mathcal{D}}
\newcommand{\KK}{\mathcal{K}}
\newcommand{\FF}{\mathcal{F}}
\newcommand{\BB}{\mathcal{B}}

\newcommand{\mfs}{\mathfrak{s}}
\newcommand{\mft}{\mathfrak{t}}
\newcommand{\mfa}{\mathfrak{a}}
\newcommand{\mfc}{\mathfrak{c}}
\newcommand{\mfb}{\mathfrak{b}}
\newcommand{\mfd}{\mathfrak{d}}

\newcommand{\hT}{\hat{T}}
\newcommand{\hS}{\hat{S}}

\newcommand{\tS}{\tilde{S}}

\newcommand{\one}{\mathbf{1}}
\newcommand{\Q}{\mathbf{Q}}
\newcommand{\Z}{\mathbb{Z}}

\newtheorem{theorem}{Theorem}[section]

\newtheorem*{theorem*}{Theorem}
\theoremstyle{definition}

\newtheorem{example}[theorem]{Example}

\theoremstyle{remark}
\newtheorem{remark}[theorem]{Remark}

\numberwithin{equation}{section}

\begin{document}

\title[Near-group centers and super-modular categories]{On near-group centers and super-modular categories}


\author{Eric C. Rowell}
\address{Department of Mathematics\\
    Texas A\&M University\\
    College Station, TX 77843-3368\\
    U.S.A.}
\curraddr{}
\email{rowell@math.tamu.edu}
\thanks{}

\author{Hannah Solomon}
\address{Department of Mathematics\\
    Texas A\&M University\\
    College Station, TX 77843-3368\\
    U.S.A.}
\curraddr{}
\email{hmsolomon@tamu.edu}
\thanks{}

\author{Qing Zhang}
\address{Department of Mathematics\\
    Purdue University\\
    West Lafayette, IN 47907\\
    U.S.A.}
\curraddr{}
\email{zhan4169@purdue.edu}
\thanks{}
\subjclass[2020]{Primary}

\date{}

\begin{abstract} The construction and classification of super-modular categories is an ongoing project, of interest in algebra, topology and physics. In a recent paper, Cho, Kim, Seo and You produced two mysterious families of super-modular data, with no known realization.
We show that these data are realized by modifying the Drinfeld centers of near-group fusion categories associated with the groups $\mathbb Z/6$ and $\mathbb Z/2\times \mathbb Z/4$.  The methods we develop have wider applications and we describe some of these, with a view towards understanding when near-group centers provide super-modular categories.
\end{abstract}

\maketitle

\section{Introduction}\label{intro}

Besides their interest in algebraic category theory and topology, 
modular (resp. super-modular) categories are important in condensed matter physics as they describe bosonic (resp. fermionic) topological phases of matter in two spacial dimensions (see eg. \cite{RSW,16fold}).   This connection goes back to the mathematical study of conformal field theory in \cite{MooreSieberg,AMV}.  It is of substantial interest to classify these categories both for these applications and their intrinsic beauty.

It is known \cite{BNRW1,JMNR} that for any $r$ there are finitely many modular (resp. super-modular) categories with precisely $r$ isomorphism classes of simple objects, i.e. \textbf{rank} $r$.  Although a complete classification of modular (resp. super-modular) categories is probably out of reach without some general structure theorems, 
 for small $r$, classifications are known \cite{RSW,BNRW2,NRWW,bruillard2020classification,PPRZ}, at least up to \emph{modular data}.

From any modular category $\CC$ one obtains a (projective) representation of the modular group $\operatorname{SL}(2,\Z)$ as the mapping class group of the torus.  This representation is determined by a pair of matrices $(S,T)$ called the {modular data} of $\CC$.  The matrices $S$ and $T$ satisfy a number of remarkable constraints, including the key result of \cite{NS10} that says that the $\operatorname{SL}(2,\Z)$ representation factors through $\operatorname{SL}(2,\Z/N)$ for some minimal $N$, called the \textbf{level} of the representation, and that $N$ is the (finite) order of the $T$-matrix.
The paper \cite{NRWW} introduced a computational method for classifying (low) rank $r$ modular categories roughly as follows:
\begin{enumerate}
    \item[Step 1] Construct all irreducible prime-power-level representations of $\operatorname{SL}(2,\Z)$ of dimension at most $r$, with the property that the image of $\mfs:=\begin{pmatrix} 0 & -1\\ 1 & 0\end{pmatrix}$ is symmetric and the image of $\mft:=\begin{pmatrix} 1 & 1\\ 0 & 1\end{pmatrix}$   is diagonal.  That is, all such representations factoring over $\operatorname{SL}(2,\Z/p^k)$ for primes $p$ and $k\geq 1$.  
    \item[Step 2] Construct all finite-level representations $\rho$ of dimension $r$ by considering direct sums of tensor products of representations from Step 1.
    \item[Step 3] Construct possible modular data $(S,T)$ of rank $r$ by studying matrices of the form $S=U^{-1}\rho(\mfs) U$ where $U$ is orthogonal \cite[Theorem 3.4]{NRWW} and commutes with $\rho(\mft)$ for each $\rho$ from Step 2.
    \item[Step 4] Use the numerous constraints on modular data to eliminate as many pairs $(S,T)$ as possible.
    \item[Step 5] For each remaining pair $(S,T)$ find a modular category $\CC$ with this modular data.
\end{enumerate}
In \cite{NRWW} this approach was successfully applied to classify rank $6$ modular categories up to modular data.\footnote{In principle there could be inequivalent categories with the same modular data, but they would of course have the same fusion rules so the ambiguity is modest.} The number of cases to be considered proliferates rapidly, which requires significant assistance from computational software.

Super modular categories are slight generalizations of modular categories. Modular categories $\CC$ are non-degenerate in the sense that the symmetric center $\operatorname{Sym}(\CC)$ is the trivial category $\operatorname{Vec}$, while super-modular categories $\CC$ are \textbf{slightly degenerate} \cite{DNO}: $\operatorname{Sym}(\CC)$ is equivalent to the braided fusion category $\operatorname{sVec}$ of super-vector spaces. A technical, but inessential, additional assumption is that a super-modular category should also be unitary.  While super-modular categories have $S$ and $T$ matrices, they do not immediately yield a representation of a group, since $S$ will be degenerate.  However, after an appropriate reduction one obtains a representation of the index 3 subgroup $\Gamma_\theta\subset \operatorname{SL}(2,\Z)$ generated by $\mfs$ and $\mft^2$.  The idea is that both $S$ and $T^2$ have a well-defined tensor-decomposition into $S=\hat{S}\otimes  \frac{1}{\sqrt{2}}\begin{pmatrix}
    1 & 1\\1&1
\end{pmatrix}$ and $T^2=\hat{T}^2\otimes \begin{pmatrix}
    1 &0\\ 0&1
\end{pmatrix}$ where $\hat{S}$ and $\hat{T}^2$ give a projective representation of $\Gamma_\theta$. The pair $(\hat{S},\hat{T}^2)$ is called the \textbf{super-modular data} of $\CC$ -- note that $\hat T$ is only defined up to sign choices.

  A natural problem is to extend the above-described approach for classifying low rank modular categories to super-modular categories.  The crucial Step 1 is justified by the recently proved Minimal Modular Extension (MME) Theorem of \cite{DaveandTheo} and \cite{BRWZ} which together show that the representations of $\Gamma_\theta$ coming from super-modular categories factor over the finite group $\Gamma_\theta\cap \operatorname{SL}(2,\Z/N)$ for some $N$.  Then one can prove appropriate modifications of the steps above for the super-modular setting.  As a more modest goal one might hope to produce new super-modular data to aid in the classification.

  Under certain assumptions, Steps 1-4 were used in \cite{cho2022classification} to carry out a partial classification of super-modular data of rank $8$ and rank $10$.  The expectation is that the rank $8$ classification is complete, as it coincides with the partial classification in \cite{PPRZ}.  The authors of \cite{cho2022classification} produced 2 families of super-modular data $(\hat{S},\hat{T}^2)$ of rank 10 that were not known to have a realization, see (\ref{SMDS1}) and (\ref{SMDS2}).

  The main motivation for this article is to realize the super-modular data (\ref{SMDS1}) and (\ref{SMDS2}) by finding super-modular categories  with these super-modular data.  To do this we must find modular categories containing a fermion $f$ (i.e. \textbf{spin} modular categories \cite{16fold})  so that the centralizer $C_{\CC}(\langle f\rangle)$ of $f$ (also denoted by $\langle f\rangle^\prime$) is a super-modular category with the  given data.  Here by a \textbf{fermion} we mean an object $f$ with $f\otimes f\cong \one$ that has self-braiding $c_{f,f}=-\operatorname{Id}_{f\otimes f}$.  If such a spin modular category exists there is a 16-fold ambiguity, since there are precisely 16 minimal modular extensions of any given super-modular category \cite{16fold,LanKongWen}. Our main insight is that the super-modular data in \cite{cho2022classification} bears some similarity with modular data found in \cite{grossman2020infinite}, which is conjecturally the modular data associated with the Drinfeld centers of near-group fusion categories \cite{Siehler} (see Section \ref{NGcat}). In particular, we find that near-group categories associated with the groups $\Z/6$ and $\Z/4\times \Z/2$ yield such categories, after taking their Drinfeld centers, condensing a boson in the second case, and discarding pointed modular factors. The work of \cite{EvansGannon} and \cite{izumi2001structure} lay the groundwork for our approach, with the main difficulty being producing the modular data of these Drinfeld centers -- this involves solving a large system of non-linear equations.  We have the following:
  \begin{theorem*}
      \begin{enumerate}
          \item[(a)] Let $\mathcal C$ be the Drinfeld center of a near-group category of type $\mathbb Z/6+6$. Then $\mathcal C\cong \mathcal D\boxtimes \mathcal C(\mathbb Z/3, q)$, where $\mathcal D$ is a spin modular category, and $q$ is its associated quadratic form restricted to $\mathbb Z/3$. Moreover, the Müger centralizer of the fermion $f$ in $\mathcal D$ is super-modular and either itself or one of its Galois conjugates has the same super-modular data as in (\ref{SMDS1}). 
          \item[(b)] Let $\mathcal{C}$ be the Drinfeld center of a near-group category of type $\mathbb Z/2\times \mathbb Z/4 +8$. Then $\left[\mathcal{C}_{\mathbb \Z/2}\right]_0\cong \mathcal{D}\boxtimes \mathcal{C}(\mathbb Z/2, q)$, where $\mathcal{D}$ is a spin modular category and $q$ is the associated quadratic form restricted to $\mathbb Z/2$. Moreover, the Müger centralizer of the fermion $f$ in $\mathcal D$ is super-modular and either itself or one of its Galois conjugates has the same super-modular data as in (\ref{SMDS2}). 
      \end{enumerate}
  \end{theorem*}
  We remark that for case $(b)$ we must first \emph{condense a boson}, i.e. take the $\Z/2$-de-equivariantization with respect to the symmetric Tannakian category $\langle b\rangle$ where $b$ is a \textbf{boson}: an object $b$ so that $b\otimes b\cong \one$ and the self-braiding satisfies $c_{b,b}=\operatorname{Id}_{b\otimes b}$.

  Along the way we noticed that our approach is quite general: we can often construct spin and hence super-modular categories from Drinfeld centers of near-group categories (for groups of even order).  We illustrate this with some examples.

  \textbf{Acknowledgments} E.R. and H.S. were partially supported by NSF grants DMS-2000331 and DMS-2205962.  The authors thank T. Gannon, A. Schopieray, Y. Wang, A. Bagheri and P. Gustafson for enlightening conversations.

\section{Preliminaries}
We shall need a few standard, but technical notions from the theory of braided fusion categories.
 Throughout this paper, we use the notation $\zeta_n:=e^{2 \pi i/n}$, $\mathbb{T}=\{x \in \mathbb{C} ;|x|=1\}$, and  $\chi_n^m=m+\sqrt{n}$.  We occasionally employ the shorthand $XY:=X\otimes Y$ for notational convenience, and also write $X\oplus Y$ as $X+Y$.  {We denote by $e(r)=exp(2\pi i r)$ for any rational number $r$.}

A braiding on a fusion category $\mathcal C$ is a natural isomorphism $c_{X, Y}: X \otimes Y  \rightarrow Y\otimes X$ satisfying the hexagon equations \cite{EGNO}. A braided fusion category is a fusion category equipped with a braiding. We call a braided fusion category $\mathcal C$ \textbf{symmetric} if $c_{Y, X}c_{X, Y}=\operatorname{Id}_{X \otimes Y}$ for all $X, Y \in \mathcal{C}$. Let $\mathcal C$ be a braided fusion category and $\mathcal{D}\subset \mathcal C$ a fusion subcategory.  The Müger centralizer $C_{\mathcal C}(\mathcal D)$ of $\mathcal D$ in $\mathcal C$ is the symmetric fusion subcategory of $\mathcal C$ generalized by $X\in \mathcal C$ such that $c_{Y, X}  c_{X, Y}=\operatorname{Id}_{X \otimes Y}$ for all $Y\in \mathcal D$. We will often use the shorthand notation $\DD^\prime=C_\CC(\DD)$ when no confusion can arise. The M\"uger center of $\mathcal C$ is the fusion subcategory $C_{\mathcal C}(\mathcal C)=\CC^\prime$, which is also called the symmetric center and sometimes denoted $\operatorname{Sym}(\mathcal C)$. A premodular category is a spherical braided fusion category (in other notation, a ribbon fusion category). A premodular category $\mathcal C$ is called \textbf{modular} if  $\operatorname{Sym}(\mathcal C)\cong \operatorname{Vec}$, that is, $\operatorname{Sym}(\mathcal C)$ is equivalent to the category  of finite-dimensional vector spaces over $\mathbb C$. For a premodular category $\CC$,  the $\tS$ matrix is defined by $\tS_{X,Y}=\operatorname{Tr}(c_{Y,X^*}c_{X^*,Y})$ and the $T$ matrix by $T_{X,Y}=\delta_{X,Y}\theta_X$, where $\theta_X$ is the (scalar form of the) ribbon twist. The $S$-matrix is a normalized version of $\tS$, namely $S:=\frac{\tS}{\sqrt{\dim(\CC)}}$.  An alternative definition of a modular category is a premodular category whose $S$ matrix is non-degenerate. A \textbf{super-modular} category $\mathcal B$ is a premodular category with $\operatorname{Sym}(\mathcal{B})$ equivalent to the category $\operatorname{ sVec}$ of super-vector spaces.
{In this case, the $S$ matrix is degenerate.  For any super-modular category it is more convenient to consider $(\hat{S},\hat{T}^2)$, where $S=\hat{S}\otimes \frac{1}{\sqrt{2}}\begin{pmatrix}
    1 & 1\\1&1
\end{pmatrix}$ and $T^2=\hat{T}^2\otimes \begin{pmatrix}
    1 &0\\ 0&1
\end{pmatrix}$.  Then, $\hat{S}$ is non-degenerate.}
A super-modular category $\CC$ is called \textbf{split} if there is a modular category $\DD$ such that $\CC\cong \DD\boxtimes \sVec$ as braided fusion categories.  Clearly non-split super-modular categories are of greatest interest.  
The following recent result is crucial:

\begin{theorem}\cite{DaveandTheo} Let $\CC$ be any super-modular category.  Then there exists a (psuedo-unitary) modular category $\DD$ such that $\CC\subset\DD$ and $\dim(\DD)=2\dim(\CC)$.
\end{theorem}
As the minimal dimension of a (pseudo-unitary) modular category containing a super-modular category $\CC$ is $2\dim(\CC)$, such a category is called a \textbf{minimal modular extension} of $\CC$. A modular category $\DD$ with a fermion $f$ is called a spin modular category.  In this case $\DD=\DD_0\oplus\DD_1$ is $\Z/2$-graded, where the trivial component $\DD_0=C_{\DD}(\langle f\rangle)$. By results of \cite{16fold,LanKongWen} there are exactly $16$ minimal modular extensions of any super-modular category. 
Clearly if $\CC=\DD\boxtimes \sVec$ is a split super-modular category then $\DD\boxtimes \KK$ is a minimal modular extension where $\KK$ is one of the $16$ minimal modular extensions of $\sVec$ \cite{kitaev16fold}. Thus a super-modular category  $\CC$ is split if and only if any (hence, every) minimal modular extension factors in this way.

{We recall some results and notation from  \cite[Section 8.4]{EGNO}. A braided fusion category is \textbf{pointed} if every simple object is invertible.  In a pointed braided fusion category, the isomorphism classes of simple objects form a finite abelian group. 
For a finite abelian group $G$,  a \textbf{quadratic form} on $G$ is a map $q:G\rightarrow \mathbb{F}^\times$ such that $q(g)=q(-g)$ and the map $\langle\;,\;\rangle:G\times G\rightarrow\mathbb{F}^\times$ with $\langle g,h\rangle=\frac{q(g+h)}{q(g)q(h)}$ is a symmetric bicharacter. 
Let $G$ be a finite abelian group and $q:G\rightarrow \mathbb{F}^\times$ a quadratic form on $G$.  The pair $(G,q)$ is called a \textbf{pre-metric group}.
Consider the pointed braided fusion category with (isomorphism classes of) simple objects labeled by $g\in G$ and braiding $c_{g,h}$.  We can define a quadratic form $q:G\rightarrow\mathbb{F}^\times$ on $G$ by sending $g$ to $c_{g,g}\in\mathbb{F}^\times$.  Then, $(G,q)$ is a pre-metric group.  In fact, up to braided equivalence, for each pre-metric group $(G,q)$, there is a unique pointed braided fusion category, denoted by $\mathcal{C}(G,q)$.
By equipping $\mathcal{C}(G,q)$ with the spherical structure $\theta_g=q(g)$, we get a premodular category with a symmetric $S$ matrix defined by $S_{g,h}=\frac{1}{\sqrt{|G|}}\overline{\langle g,h\rangle}$ and a diagonal $T$ matrix defined by $T_{g,g}=q(g)$ (we use the conventions of \cite{grossman2020infinite}).
If $\langle\;,\;\rangle$ is a non-degenerate bicharacter, then $q$ is a \textbf{non-degenerate} quadratic form.  In this case, the pair $(G,q)$ is called a \textbf{metric group}.  The corresponding pointed category for a metric group, $\mathcal{C}(G,q)$, is then a modular category with $S$ and $T$ as its modular data. }

\subsection{$G$-de-equivariantization and Boson condensation} \label{subsec:anyoncondensation}

Suppose that $\CC$ is a braided fusion category with a subcategory equivalent to $\Rep(G)$.  In \cite{Bruguieres,muger2000galois,dgno}, three related concepts are described:  modularization, modules over an algebra object and de-equivariantization.  The $G$-de-equivariantization $\CC_G$ of $\CC$ is $G$-graded, and the trivial component $[\CC_G]_0$ is again braided.  
Boson condensation is the term used in the physical literature for the process $\CC\rightsquigarrow [\CC_G]_0$. This has a non-unique reverse process known as gauging \cite{CGPW}. There is an alternative description of boson condensation in terms of algebra objects that is sometimes more useful: the object $A=\operatorname{Fun}(G)\in\Rep(G)\subset\CC$ has the structure of an algebra object in $\CC$.  The category of $A$-modules in $\CC$, denoted $\CC_A$ is equivalent to $\CC_G$, and the category of so-called \emph{local} $A$-modules $\CC_A^0$ coincides with $[\CC_G]_0$.  
If $\CC$ is modular then so is $[\CC_G]_0$. Moreover, in this case $\operatorname{Sym}(\Rep(G)^\prime)=\Rep(G)$ and $[\CC_G]_0$ is equivalent to $(\Rep(G)^\prime)_G$: the $G$-de-equivariantization of the centralizer of $\Rep(G)$ inside $\CC$. This is called the \textbf{modularization} of $\Rep(G)^\prime$ \cite{Bruguieres}.  We remark that many algebra objects not of the form $\operatorname{Fun}(G)$ exist and leads to a more general construction.

Here we describe a practical computational approach to boson condensation in the special case of $G=\Z/2$ which is all we will need.  The non-trivial simple object in $\Rep(\Z/2)$ is a boson, which we denote $b$.  The corresponding algebra $A=\one+b$ is an object. In this case for any simple object $X$ we have either $bX\cong X$ or $bX\cong Y$ and $bY\cong X$ for some simple  object $Y\not\cong X$.   Let $F$ be the $\Z/2$-de-equivariantization functor restricted to the subcategory $\langle b\rangle ^\prime$ (i.e. the centralizer of the boson $b$) so that we have $F:\langle b\rangle ^\prime\rightarrow \CC_A^0$.  For an $X\in \langle b\rangle ^\prime$, if $bX\cong Y\not\cong X$ then $F(X)\cong F(Y)$ is a simple object in $\CC_A^0$ with dimension $\dim(X)$.  We can denote the simple for $F(X)\cong F(Y)$ by $\alpha$, and call it type I.  If $bX\cong X$ then $F(X)\cong \alpha_1\oplus \alpha_2$ where $\alpha_i$ are simple objects in $\CC_A^0$ with dimension $\dim(X)/2$.  This is called type II.  For $X\in \langle b\rangle ^\prime$ we have that $\theta_{F(X)}=\theta_X$.

We would like to compute the $\tS$-matrix entries of $\CC_A^0$ in the special case that $A=\one+ b$ for a boson $b$.  More general results are known, see \cite{BagheriThesis}.   For our purposes the following somewhat \emph{ad hoc} approach will suffice.

Recall that the balancing equation, for simple objects $X,Y$ is
\begin{equation}\label{balance}
    \theta_X\theta_Y\tS_{X,Y}=\sum_{Z\in\Irr(\CC)} N_{X^*,Y}^Z\theta_Zd_Z
\end{equation}
where the sum is over simple objects $Z$, $N_{X^*,Y}^Z=\dim\Hom(X^*\otimes Y,Z)$ and $d_Z$ is the categorical dimension of $Z$. From this
we can infer further information about the $\tS$-matrix of $\CC_A^0$.
\begin{theorem}\label{thm:s-matrixcondense}
    \begin{enumerate} Suppose $b\in\CC$ is a boson and $F$ is the condensation functor $F:\langle b\rangle ^\prime\rightarrow \CC_A^0$.  Let $X,Y\in\langle b\rangle^\prime$.
        \item[(a)] Suppose that $F(X)=\alpha$ and $F(Y)=\beta$ are simple, i.e., type I simple objects.  Then $\tS_{\alpha,\beta}=\tS_{X,Y}$.
        \item[(b)] Suppose that $F(X)=\alpha$ is simple and $F(Y)\cong\beta_1+\beta_2$ with $\beta_1, \beta_2$ simple, i.e., $X$ is type I and $Y$ is type II.
        Then $\tS_{\alpha,\beta_1}+\tS_{\alpha,\beta_2}=\tS_{X,Y}$.
    \end{enumerate}
\end{theorem}

\begin{proof}

Recall \cite{Bruguieres,muger2000galois} that $$\Hom(F(X) \otimes F(Y),F(Z))=\Hom(X \otimes Y,Z\otimes (\one+ b)).$$  Suppose that $bX\not\cong X$ and $bY\not\cong Y$ are simple objects in $\langle b\rangle^\prime$, so that $F(X)=\alpha$ and $F(Y)=\beta$ are simple objects in $\CC^0_A$.  Consider the $\tS$-matrix entry $\tS_{\alpha,\beta}$.  By (\ref{balance}) we have:

$$\tS_{\alpha,\beta}=\frac{1}{\theta_{\alpha}\theta_{\beta}}\sum_{\gamma\in \Irr(\CC_A^0)}N_{(\alpha)^*,\beta}^\gamma d_\gamma\theta_\gamma .$$
There are two cases to consider:
\begin{enumerate}
    \item There is a simple $Z\in\CC$ such that $F(Z)=\gamma$.  In this case $bZ\not\cong Z$, and $d_\gamma=d_Z$ and $\theta_\gamma=\theta_Z$.
    \item There is a simple $Z\in\CC$ such that $F(Z)\cong \gamma_1+ \gamma_2$ with $\gamma_i$ simple.  In this case $bZ\cong Z$ and we have $\theta_{\gamma_1}=\theta_{\gamma_2}=\theta_Z$ and $d_{\gamma_1}=d_{\gamma_2}=\frac{d_Z}{2}$.
\end{enumerate}
In the first case we see that $N_{(\alpha)^*,\beta}^\gamma d_\gamma\theta_\gamma=N_{X^*,Y}^Zd_Z\theta_Z+N_{X^*,Y}^{bZ}d_{bZ}\theta_{bZ}$ with $Z\not\cong bZ$.  In the second case we find that $N_{(\alpha)^*,\beta}^{\gamma_1} =N_{(\alpha)^*,\beta}^{\gamma_2}=N_{X^*,Y}^Z$, so that $N_{X^*,Y}^Zd_Z\theta_Z=N_{(\alpha)^*,\beta}^{\gamma_1}d_{\gamma_1}\theta_{\gamma_1}+N_{(\alpha)^*,\beta}^{\gamma_2}d_{\gamma_2}\theta_{\gamma_2}$.  Thus we find that 
$$\tS_{\alpha,\beta}=\sum_{\gamma\in \Irr(\CC_A^0)}N_{(\alpha)^*,\beta}^\gamma d_\gamma \frac{\theta_\gamma}{\theta_{\alpha}\theta_{\beta}}=\sum_{Z\in\Irr(\CC)} N_{X^*,Y}^Zd_Z\frac{\theta_Z}{\theta_X\theta_Y}=\tS_{X,Y}$$
proving (a).

Now consider the case of simple objects $X,Y$ in $\langle b\rangle ^\prime$ such that $bX\not\cong X$ and $bY\cong Y$.  From this we have that $F(X)=\alpha$ and $F(Y)\cong\beta_1+ \beta_2$ where $\beta_i$ are non-isomorphic simple objects of dimension $\dim(\beta_i)=\frac{\dim(Y)}{2}$.    We argue similarly as above to see the following:

\begin{enumerate}
    \item If $\gamma=F(Z)$ comes from an object $Z\in\Irr(\CC)$ such that $bZ\not\cong Z$ we have: $$N_{X^*,Y}^Z+N_{X^*,Y}^{bZ}=N_{\alpha^*,\beta_1}^{\gamma}+N_{\alpha^*,\beta_2}^{\gamma},$$ while
    \item if $\gamma_1+ \gamma_2\cong F(Z)$ for $Z\in\Irr(\CC)$ we find that 

$$N_{\alpha^*,\beta_1}^{\gamma_1}+N_{\alpha^*,\beta_2}^{\gamma_1}+N_{\alpha^*,\beta_1}^{\gamma_2}+N_{\alpha^*,\beta_2}^{\gamma_2}=2N_{X^*,Y}^Z.$$
\end{enumerate}
Since $Z\in\Irr(\CC)$ is of type I if $F(Z)=\gamma$ is simple, i.e. if $ bZ\not\cong Z$, and type II if $F(Z)\cong\gamma_1+\gamma_2$, i.e. $bZ\cong Z$, we can partition $\Irr(\CC)=\Irr(\CC)_I\cup b\Irr(\CC)_I\cup \Irr(\CC)_{II}$ where we have chosen representatives of each $\Z/2$-orbit for objects of type I so that every type I object is in exactly one of $\Irr(\CC)_I$ or $b\Irr(\CC)_I$.   Recall that $\theta_{\beta_i}=\theta_Y$, $\theta_{\alpha}=\theta_X$, $\theta_\gamma=\theta_Z$, $d_\gamma=d_Z$, and $d_{\gamma_i}=\frac{d_Z}{2}$.
We compute:

\begin{eqnarray*}
    \tS_{\alpha,\beta_1}+\tS_{\alpha,\beta_2}&=&\sum_{\gamma\in \Irr(\CC_A^0)} (N_{\alpha^*,\beta_1}^{\gamma}+N_{\alpha^*,\beta_2}^{\gamma})\frac{d_Z\theta_Z}{\theta_X\theta_Y}=\\
   \sum_{Z\in\Irr(\CC)_I} (N_{X^*,Y}^{Z}+N_{X^*,Y}^{bZ})\frac{d_Z\theta_Z}{\theta_X\theta_Y}&+ &\sum_{\stackrel{Z\in\Irr(\CC)_{II}}{F(Z)=\gamma_1+\gamma_2}} (N_{\alpha^*,\beta_1}^{\gamma_1}+N_{\alpha^*,\beta_2}^{\gamma_1}+N_{\alpha^*,\beta_1}^{\gamma_2}+N_{\alpha^*,\beta_2}^{\gamma_2})\frac{d_Z\theta_Z}{2\theta_X\theta_Y}=\\
  && \sum_{Z\in\Irr(\CC)_I} (N_{X^*,Y}^{Z}+N_{X^*,Y}^{bZ})\frac{d_Z\theta_Z}{\theta_X\theta_Y}+ \sum_{Z\in\Irr(\CC)_{II}} (2N_{X^*,Y}^{Z})\frac{d_Z\theta_Z}{2\theta_X\theta_Y}=\tS_{X,Y}
\end{eqnarray*}
proving (b).
\end{proof}

\subsection{Near-group categories}\label{NGcat}

    Let $G$ be a finite group of order $n$ and $m$ a nonnegative integer. A \textbf{near-group category} of type $G+m$ is a fusion category with simple objects labeled by elements $g\in G$ and an extra simple object labeled by $\rho$ such that the fusion rules are generalized by the group operation in $G$, $g\rho=\rho g=\rho$ for all $g\in G$, and $ \rho^{\otimes2} = m\rho+\sum_{g\in G}g$. The Tambara-Yamagami categories are the near-group categories with $m = 0$, which are fully classified in \cite{TY98}.  It is known \cite[Theorem 2]{EvansGannon} that the only possible values of $m$ are $n$, $n-1$ or $m = kn$ for some nonnegative integer $k$. This paper focuses on the cases when $m=n$ since these are related to the modular data we aim to realize. 
    
\begin{theorem} \cite[Theorem 5.3]{izumi2001structure}, \cite[Corollary 5]{EvansGannon} Let $G$ be a finite abelian group of order $n$, $\langle \;,\;\rangle$ a non-degenerate symmetric bicharacter on $G$ and define $d=\dfrac{n+\sqrt{n^2+4n}}{2}$. Let  $c\in \mathbb T$, $a: G\to \mathbb T$, $b: G\to \mathbb C$ be such that 
\begin{equation}
     a(0)=1, \quad a(x)=a(-x), \quad a(x+y)\langle x, y\rangle=a(x)a(y),\quad \sum_{a\in G} a(x)=\sqrt{n}c^{-3},
\end{equation}

\begin{equation}
    b(0)=-\frac{1}{d}, \quad \sum_y \overline{\langle x, y\rangle} b(y)=\sqrt{n} c \overline{b(x)}, \quad a(x) b(-x)=\overline{b(x)},
\end{equation}
\begin{equation}
    \sum_x b(x + y) \overline{b(x)}=\delta_{y, 1}-\frac{1}{d}, \quad \sum_x b(x + y) b(x + z) \overline{b(x)}=\overline{\langle y, z\rangle} b(y) b(z)-\frac{c}{d \sqrt{n}}.
\end{equation}

Then $\langle\, ,\, \rangle$, $  c, a, b$ determine a near-group fusion category of type $G+n$. Two such categories $\mathcal{C}_1$ and $\mathcal{C}_2$ determined by $\langle\,,\,\rangle_1,c_1, a_1, b_1$ and $\langle\,,\,\rangle_2,c_2, a_2, b_2$  are equivalent as fusion categories if and only if there is $\phi \in \operatorname{Aut}(G)$ such that $\langle x, y\rangle_1=\langle\phi x, \phi y\rangle_2, a_1(x)=a_2(\phi x)$, $b_1(x)=b_2(\phi x)$ and $c_1=c_2$.
\end{theorem}

\subsection{Centers of near-group categories when $m=n$}\label{Centers} 
The modular data for the center of a near-group of type $G+n$, for $|G|=n$, is given as follows \cite{izumi2001structure}.  First, we need to find all functions $\xi: G \rightarrow \mathbb{T}$ and values $\tau \in G, \omega \in \mathbb{T}$ which satisfy 
\begin{equation}\label{eq:half1}
    \sum_g \xi(g)=\sqrt{n} \omega^2 a(\tau) c^3-n d^{-1}
\end{equation}
\begin{equation}\label{eq:half2}
    \bar{c} \sum_k b(g+k) \xi(k)=\omega^2 c^3 a(\tau) \overline{\xi(g+\tau)}-\sqrt{n} d^{-1}
\end{equation}
\begin{equation}\label{eq:half3}
    \xi(\tau-g)=\omega c^4 a(g) a(\tau-g) \overline{\xi(g)}
\end{equation}
\begin{equation}\label{eq:half4}
    \sum_k \xi(k) b(k-g) b(k-h)=c^{-2} b(g+h-\tau) \xi(g) \xi(h) \overline{a(g-h)}-c^2 d^{-1}
\end{equation}

There are $n(n+3)/2$ triples $(\xi_i, \tau_i, \omega_i)$ that satisfy the above equations.
The corresponding center has rank $n(n+3)$ with the following 4 subsets of simple objects:
\begin{enumerate}
    \item $\mathfrak{a}_g, g \in G$, $\dim(\mfa_g)=1$;
    \item $\mathfrak{b}_h, h \in G$, $\dim(\mfb_h)=d+1$;
    \item $\mathfrak{c}_{l, k}=\mathfrak{c}_{k, l}, l, k \in G, l \neq k$, $\dim(\mfc_{k,l})=d+2$
    \item $\mathfrak{d}_j$, where $j$ corresponds to a triple $\left(\xi_j, \tau_j, \omega_j\right)$, $\dim(\mfd_j)=d$.
\end{enumerate}
The $T$ and $S$ matrices are given by the following block form
\begin{equation}\label{Tmatrix-center}
    T=\operatorname{diag}\left[\langle g, g\rangle,\langle h, h\rangle,\langle k, l\rangle, \omega_j\right]
\end{equation}

\begin{equation}\label{Smatrix-center}
    S=\\\frac{1}{\lambda}\left[\begin{array}{cccc}\left\langle g, g^{\prime}\right\rangle^{-2} & ( d+1)\left\langle g, h^{\prime}\right\rangle^{-2} & ( d+2) \overline{\left\langle g, k^{\prime}+l^{\prime}\right\rangle} &  d\left\langle g, \tau_{j^{\prime}}\right\rangle \\( d+1)\left\langle h, g^{\prime}\right\rangle^{-2} & \left\langle h, h^{\prime}\right\rangle^{-2} & ( d+2) \overline{\left\langle h, k^{\prime}+l^{\prime}\right\rangle} & - d\left\langle h, \tau_{j^{\prime}}\right\rangle \\( d+2) \overline{\left\langle k+l, g^{\prime}\right\rangle} & ( d+2) \overline{\left\langle k+l, h^{\prime}\right\rangle} & S_{(k, l),\left(k^{\prime}, l^{\prime}\right)} & \mathbf 0 \\ d\left\langle\tau_j, g^{\prime}\right\rangle & - d\left\langle\tau_j, h^{\prime}\right\rangle &\mathbf 0 & S_{j, j^{\prime}}\end{array}\right],
\end{equation}
where\footnote{There is a slight difference between our eqn. (\ref{Smatrix-S44}) and that of \cite[eqn. (4.58)]{EvansGannon}: the $\tau_{j}$ and $\tau_{j^\prime}$ are switched. This yielded consistent modular data in our setting.}
\begin{equation}\label{Smatrix-S44}
    \begin{aligned} & S_{j, j^{\prime}}=\omega_j \omega_{j^{\prime}} \sum_{g \in G}\left\langle\tau_j+\tau_{j^{\prime}}+g, g\right\rangle \\ & + d \omega_j \omega_{j^{\prime}}c^6 a\left(\tau_j\right) a\left(\tau_{j^{\prime}}\right) n^{-1} \sum_{g, h \in G} \overline{\xi_j(g) \xi_{j^{\prime}}(h)\left\langle\tau_{j^\prime}-\tau_{j}+h-g, h-g\right\rangle}, \end{aligned}
\end{equation}
and 
\begin{equation}\label{Smatrix-S33}
    S_{(k, l),\left(k^{\prime}, l^{\prime}\right)}=(d+2)\left(\overline{\left\langle k, k^{\prime}\right\rangle\left\langle l, l^{\prime}\right\rangle}+\overline{\left\langle k, l^{\prime}\right\rangle\left\langle l, k^{\prime}\right\rangle}\right) .
\end{equation}

\begin{remark}  As it may be useful to other researchers, here is our approach to finding solutions $(\xi,\tau, \omega)$ to equations (\ref{eq:half1})-(\ref{eq:half4}).  We fix $\omega$ a root of unity and $\tau\in G$.
Notice that equation (\ref{eq:half2}) can be rephrased to simplify solving for $\xi$: equation (\ref{eq:half3}) implies that 
$$\overline{\xi(g+\tau)}=\xi(-g)\overline{\omega c^4 a(g+\tau) a(g)}.$$
Substituting this into equation (\ref{eq:half2}), we have 
\begin{equation}\label{2.5mod}
\omega a(\tau)\overline{a(g+\tau) a(g)} \xi(-g)-\sum_k b(g+k) \xi(k)=\frac{c \sqrt{n}}{d}.
\end{equation}

 Let $C=C(\omega,\tau)$ be the matrix indexed by $G$ with entries $C_{g,h}(\omega,\tau)=\omega a(\tau)\overline{a(g+\tau) a(g)} \delta_{g,-g}$.
 Let $B$ be the matrix such that $B_{g, k}= b(g+k), g, k\in G$. Then equation (\ref{2.5mod}) becomes the system:
 \begin{equation}\label{cbeqn}
     (C(\omega,\tau)-B)\Vec{\xi}=\Vec{z}
 \end{equation} where $\Vec{z}$ has all entries equal to $\frac{c \sqrt{n}}{d}$.  

 We consider all pairs $(\omega,\tau)$ where $\tau\in G$ and $\omega$ is a root of unity, of bounded degree \cite{Etingof02MRL,BNRW1}, and first solve the system (\ref{cbeqn}).  This gives an affine set consisting of vectors $\xi_h+\xi_p$ where $\xi_p$ is a particular solution and $\xi_h\in \operatorname{Null}(C(\omega,\tau)-B)$.  Depending on the dimension of $\operatorname{Null}(C(\omega,\tau)-B)$ we choose an appropriate number of equations (possibly 0) from (\ref{eq:half4}) to determine the free parameters and obtain the vector form of $\xi$.  We then test the corresponding triples $(\xi,\tau,\omega)$ on the remaining equations of (\ref{eq:half1})-(\ref{eq:half4}) to determine if the triple is a solution or not.  When we have found exactly $n(n+3)/2$ triples $(\xi,\tau,\omega)$ that satisfy all equations we stop and compute the $S$-matrix.
\end{remark}

\subsection{Leveraging the pointed part of near-group centers}
Suppose $\CC$ is a near-group category of type $G+n$ where $n=|G|$, corresponding to the non-degenerate symmetric bicharacter $\langle \;,\;\rangle.$  Using (\ref{Smatrix-center}) and (\ref{Tmatrix-center}) we can glean a wealth of information about the center $\mathcal{Z}(\CC)$, by examining the pointed part.

The objects $\mfa_g$ with $\mfa_g^{\otimes 2}=\one$, i.e., $2g=0$ as a group element, and $\langle g,g\rangle=1$ are bosons.
Generally, if $\mfa_g$ is an invertible object then we can determine its centralizer.  It is generated by:
\begin{enumerate}
    \item the invertibles $\mfa_h$ of dimension $1$ with $\langle g,h\rangle^{-2}=1$,
    \item the simple objects $\mfb_h$ of dimension $d+1$ with $\langle g,h\rangle^{-2}=1$,
    \item the objects $\mfc_{k,\ell}$ of dimension $d+2$ with $\langle g,k+\ell\rangle=1$, and
    \item  the objects $\mfd_j=(\xi_j,\tau_j,\omega_j)$  of dimension $d$ such that $\langle g,\tau_j\rangle=1$.
\end{enumerate}

Suppose that $|G|$ is even so that there is an element $g$ with $2g=0$.  The corresponding object $\mfa_g$ has $\mfa_g^{\otimes 2}=\one$ and is either a boson or a fermion, depending on the value of $\langle g,g\rangle\in\{1,-1\}$ (recall that $\langle g,h\rangle$ is a bicharacter, so that the value of $\langle g,h\rangle=\pm 1$).
Now if $b=\mfa_g$ is a boson we can determine a significant portion of the condensation by $\langle b\rangle\cong\Rep(\Z/2)$.  First note that $\langle g,h\rangle^{-2}=\langle 2g,h\rangle^{-1}=1$ since $2g=0$.  Thus all of the objects of $\mfa_h,\mfb_h$ are centralized by $b$.  Moreover, the objects $\mfb_h$ must have $b\otimes \mfb_h\not\cong \mfb_h$. If not, then under the condensation functor the image of $\mfb_h$ would be a sum of two simple objects of dimension $\frac{d+1}{2}$, which is not an algebraic integer: indeed $d+1$ is a unit in $\Q[d]$ so $\frac{d+1}{2}$ has norm $1/4$. On the other hand $\frac{d}{2}$ and $\frac{d+2}{2}$ are algebraic integers if and only if $4\mid n$. So, the objects $\mfd_j$ and $\mfc_{k,l}$ could be fixed under tensoring with the boson if $4\mid n$.

\begin{example}\label{z8example}

Now suppose that $G=\Z/m\times \Z/{2^s}$, where $m$ is odd and $s\geq 2$ (note that $G$ is cyclic).  Suppose that there is a near-group fusion category $\CC$ realizing the fusion rules as type $G+m2^s$.  If we write $G$ additively, the non-degenerate quadratic form $a$ on $G$ is given by $$a(x,y)=exp(\frac{ 2\varepsilon_1\pi ix^2}{m})exp(\frac{\varepsilon_2\pi i y^2}{2^{s}})$$ where $\varepsilon_1=\pm1$ and $\varepsilon_2\in\{\pm 1,\pm 5\}$.
Consider the Drinfeld center $\mathcal{Z}(\CC)$.  First notice that the $\CC(\Z/m,q)$ part of the pointed subcategory in $\mathcal{Z}(\CC)$ is modular since the $S$-matrix entries are $\langle (x_1,0),(x_2,0)\rangle^{-2}=\left(\frac{a(x_1+x_2,0)}{a(x_1,0)a(x_2,0)}\right)^{-2}$ which is non-degenerate.  So we may factor this out as it does not contribute anything useful. Thus we assume $m=0$, and consider $G=\Z/{2^s}$. The element $2^{s-1}$ has order $2$ and $\langle2^{s-1},2^{s-1}\rangle=1$ and so $\mfa_{2^{s-1}}=b$ is a boson. Since none of the invertible objects are fixed by tensoring with $b$ we find that the pointed part of $[\mathcal{Z}(\CC)_{\Z/2}]_0$ comes from the cyclic group $\Z/{2^{s-1}}$, with quadratic form $A(x)=\langle x,x\rangle=exp(\pi i \varepsilon_2 x^2/2^{s-1} )$, which is non-degenerate.  Thus, the pointed part of $[\mathcal{Z}(\CC)_{\Z/2}]_0$ is modular so that $[\mathcal{Z}(\CC)_{\Z/2}]_0\cong \CC(\Z/2^{s-1},Q)\boxtimes\DD$ where $\DD$ is a modular category with trivial pointed part.

In particular we see that no non-split super-modular categories are obtained from near-group categories associated with $\Z/2^s$ for $s\geq 2$. 
   
\end{example}

\begin{example}
    Now suppose that $G=\Z/2m$ with $m$ odd and there is a near-group category $\CC$ with fusion rules as type $G+2m$. In this case a similar computation reveals that $\mathcal{Z}(\CC)\cong \CC(\mathbb Z/m,q)\boxtimes \DD$ where $\DD$ is spin modular.  Moreover, $\DD$ does not contain a minimal modular extension of $\sVec$, and hence $\DD_0$ is non-split super-modular with simple objects of dimension $1$, $d=m+\sqrt{m^2+2m}$, $d+1$ and $d+2$.  We shall see this explicitly in the case of $m=3$ below.
\end{example}

\section{Realizing modular data}

In \cite{cho2022classification}, two families of unrealized super-modular data are constructed from the congruence representations of the $\Gamma_{\theta}$ group, as outlined in Section \ref{intro}. Below are two representatives of them, where  $\chi_n^m=m+\sqrt{n}$. 
\begin{equation}\label{SMDS1}
\hat{S} = \frac{1}{\sqrt{30 \chi_{15}^4}}
\begin{bmatrix}
1 & \chi_{15}^4 & \chi_{15}^5 & \chi_{15}^3 & \chi_{15}^3 \\
\chi_{15}^4 & 1 & \chi_{15}^5 & -\chi_{15}^3 & -\chi_{15}^3 \\
\chi_{15}^5 & \chi_{15}^5 & -\chi_{15}^5 & 0 & 0 \\
\chi_{15}^3 & -\chi_{15}^3 & 0 & \frac{1}{2} \chi_5^1 \chi_{15}^3 & -\frac{2 \sqrt{30 \chi_{15}^4}}{\chi_5^5} \\
\chi_{15}^3 & -\chi_{15}^3 & 0 & -\frac{2 \sqrt{30 \chi_{15}^4}}{\chi_5^5} & \frac{1}{2} \chi_5^1 \chi_{15}^3
\end{bmatrix},\qquad \hat{T}^2 = \diag[1, 1,  e^{2 i \pi / 3}, e^{4i\pi/5}, e^{-4i\pi/5}].
\end{equation}
\begin{equation}\label{SMDS2}
\hat{S} = \frac{1}{2 \sqrt{6 \chi_{24}^5}}
\begin{bmatrix}
1 & \chi_{24}^5 & \chi_6^3 & \chi_6^3 & \chi_{24}^4 \\
\chi_{24}^5 & 1 & \chi_6^3 & \chi_6^3 & -\chi_{24}^4 \\
\chi_6^3 & \chi_6^3 & -\chi_6^3-i \sqrt{6 \chi_{24}^5}&-\chi_6^3+i \sqrt{6 \chi_{24}^5} & 0 \\
\chi_6^3 & \chi_6^3 & -\chi_6^3+i \sqrt{6 \chi_{24}^5}&-\chi_6^3-i \sqrt{6 \chi_{24}^5} & 0 \\
\chi_{24}^4 & -\chi_{24}^4 & 0 & 0 & \chi_{24}^4
\end{bmatrix},\qquad \hat{T}^2 = \diag[1, 1, -1, -1, e^{2 i \pi / 3}].
\end{equation}

In \cite[Table 6]{cho2022classification} we find a family of $4$ different data related to (\ref{SMDS1}) and a pair of data related to (\ref{SMDS2}) with positive dimensions. They are related to each other by Galois automorphisms, so we will be satisfied with realizing one set of data for each family.

Observing that $\frac{6+\sqrt{6^2+24}}{2}=3+\sqrt{15}=\chi^3_{15}$  and $\frac{8+\sqrt{8^2+32}}{2}=4+\sqrt{24}=\chi_{24}^4$ it is reasonable to expect that this super-modular data is related to that of the centers of near-group fusion categories coming from groups of order $6$ and $8$.
We will show that this is indeed the case: for $G=\Z/6$ and $G=\Z/2\times\Z/4$.  We point out that by Example \ref{z8example} $G=\Z/8$ will not yield a non-split super modular category, whereas the near-group fusion rule associated with $G=(\Z/2)^3$ is not realizable by results of \cite{schoppreprint}.

\subsection{$G=\mathbb Z/6$} 
There are 4 inequivalent near-group fusion categories \cite{EvansGannon, Budthesis} of type $\mathbb Z/6+6$.  Their basic data is summarized in Table \ref{table:1}. The nondegenerate symmetric bicharacter  for a cyclic group $\mathbb Z/n$ is in the form of $\langle x, y\rangle=\exp (2 \pi i m x y / n)$, where $m\in \mathbb Z$ and $\gcd(m,n)=1$. When $n$ is even, the function $a$ is in the form of $a(x)=\varepsilon^x \exp \left(-\pi i m x^2 / n\right)$, where $\varepsilon=\pm 1$. By taking $m \in\mathbb Z/2n$ and coprime to $n$, we can eliminate the potential factor of $-1$.  This $m$ is recorded in the second column of Table \ref{table:1}. As $b(0)=-\frac{1}{d}$, $b(-x)=\overline{a(x) b(x)}$ and $b(x)=\frac{1}{\sqrt{n}} \exp (i j(x))$, where $j(x)\in [-\pi, \pi]$,  it suffices to list the values of $j(1),j(2)$ and $j(3)$ to determine $b(x)$.

\begin{table}[h]
\begin{tabular}{|l|c|l|l|}
\hline
\# & \multicolumn{1}{l|}{$m$} & \multicolumn{1}{c|}{$c$} &\qquad \qquad$ j(1), j(2), j(3) $ \\ \hline
$J_6^1$  & $5$                      & $\zeta_{24}^{5}$         &  $\,\,\,\,2.91503,-1.59091,\,\,\,\,2.35619$   \\ 
$\overline{J_6^1}$& $-5$                     & $\zeta_{24}^{-5}$        &  $-2.91503,\,\,\,\,1.59091,-2.35619$   \\ \hline
$J_6^2$  & $1$                      & $\zeta_{24}^{1}$         & $\,\,\,\,  2.95526,\,\, \,\,  0.0553542,-0.785398$    \\ 
$\overline{J_6^2}$  & $-1$                     & $\zeta_{24}^{-1}$        &  $-2.95526,-0.0553542,\,\,\,\,0.785398$   \\ \hline

\end{tabular}
\caption{The data for near-group categories of type $\mathbb Z/6+6$.}
\label{table:1}
\end{table}

We consider the near-group $J_6^1$ in Table \ref{table:1} and solve equations (\ref{eq:half1}) - (\ref{eq:half4}). The solutions  $(\omega_j, \tau_j, \xi_j)$ are listed in Table \ref{table:2}. The $\omega_j$'s are equal to $\zeta_{60}^k$ for some integer $k$, and we list the values for $k$ in the corresponding column. As $\xi_j(g)=\exp(i\theta_{j, g})$, we record the values $\theta_{j, g}$ in the column for $\xi$ with the order $g=0,\dots, 5\in \mathbb Z/6$.

\begin{table}[h]
\begin{tabular}{|l|l|l|l|}
\hline
\# & $\omega$ & $\tau$ & \multicolumn{1}{c|}{$\xi$}                                     \\ \hline
1  & $12$      & $0$    & $-3.03687,1.02812,-0.75497,1.67552,-1.12999,0.228519$       \\
2  & $18$     & $0$    & $0.418879,-0.476051,1.75517,1.98968,-3.0118,2.36101$ \\
3  & $42$     & $0$    & $1.67552,-1.12999,0.228519,-3.03687,1.02812,-0.75497$       \\
4  & $48$     & $0$    & $1.98968,-3.0118,2.36101,0.418879,-0.476051,1.75517$ \\ \hline
5  & $23$     & $1$    & $-2.58859,1.33196,1.6366,-0.961707,-0.29493,-0.798841$       \\
6  & $23$     & $1$    & $-0.961707,-0.29493,-0.798841,-2.58859,1.33196,1.6366$      \\
7  & $47$     & $1$    & $3.06462,-1.80798,-1.19626,-0.601743,1.85838,-1.73589$      \\
8  & $47$     & $1$    & $-0.601743,1.85838,-1.73589,3.06462,-1.80798,-1.19626$    \\
9  & $35$     & $1$    & $2.69346,-2.69346,1.0472,2.69346,-2.69346,1.0472$    \\ \hline
10 & $2$      & $2$    & $0.228519,-3.03687,1.02812,-0.75497,1.67552,-1.12999$ \\
11 & $8$     & $2$    & $2.36101,0.418879,-0.476051,1.75517,1.98968,-3.0118$       \\
12 & $32$     & $2$    & $-0.75497,1.67552,-1.12999,0.228519,-3.03687,1.02812$ \\
13 & $38$     & $2$    & $1.75517,1.98968,-3.0118,2.36101,0.418879,-0.476051$       \\ \hline
14 & $3$     & $3$    & $1.6366,-0.961707,-0.29493,-0.798841,-2.58859,1.33196$       \\
15 & $3$     & $3$    & $-0.798841,-2.58859,1.33196,1.6366,-0.961707,-0.29493$    \\
16 & $27$     & $3$    & $-1.73589,3.06462,-1.80798,-1.19626,-0.601743,1.85838$    \\
17 & $27$     & $3$    & $-1.19626,-0.601743,1.85838,-1.73589,3.06462,-1.80798$      \\
18 & $15$     & $3$    & $1.0472,2.69346,-2.69346,1.0472,2.69346,-2.69346$      \\ \hline
19 & $2$      & $4$    & $1.02812,-0.75497,1.67552,-1.12999,0.228519,-3.03687$ \\
20 & $8$     & $4$    & $-0.476051,1.75517,1.98968,-3.0118,2.36101,0.418879$       \\
21 & $32$     & $4$    & $-1.12999,0.228519,-3.03687,1.02812,-0.75497,1.67552$ \\
22 & $38$     & $4$    & $-3.0118,2.36101,0.418879,-0.476051,1.75517,1.98968$       \\ \hline
23 & $23$     & $5$    & $1.33196,1.6366,-0.961707,-0.29493,-0.798841,-2.58859$       \\
24 & $23$     & $5$    & $-0.29493,-0.798841,-2.58859,1.33196,1.6366,-0.961707$    \\
25 & $47$     & $5$    & $1.85838,-1.73589,3.06462,-1.80798,-1.19626,-0.601743$    \\
26 & $47$     & $5$    & $-1.80798,-1.19626,-0.601743,1.85838,-1.73589,3.06462$      \\
27 & $35$     & $5$    & $-2.69346,1.0472,2.69346,-2.69346,1.0472,2.69346$      \\ \hline
\end{tabular}
\caption{$(\omega_j, \tau_j, \xi_j)$ for $J_6^1$ in Table \ref{table:1}}
\label{table:2}
\end{table}

The center of the near-group category $J_6^1$ has the following simple objects:
\begin{itemize}
    \item[] six invertible objects $X_g$, $g\in\mathbb Z/6$;
    \item[] six $\chi_{15}^4$-dimensional objects $Y_h$, $h\in\mathbb Z/6$;
    \item[] fifteen $\chi_{15}^5$-dimensional objects $Z_{k,l}$, $k, l\in \mathbb Z/6$, $k< l$; and
    \item[] twenty-seven $\chi_{15}^3$-dimensional objects $W_{\omega_i,\tau_i, \xi_i}$, where $(\omega_i,\tau_i, \xi_i)$ are solutions in Table \ref{table:2}.
\end{itemize}

Applying the formulas for the modular data in Section \ref{Centers}, we have the following $T$- and $S$-matrices

$\begin{aligned} & T_{X_g}=e\left(\frac{5 g^2}{6}\right), \quad T_{Y_h}=e\left(\frac{5 h^2}{6}\right), \quad T_{Z_{k, l}}=e\left(\frac{5 k l}{6}\right), g, h, k, l \in \mathbb{Z} / 6, k<l, \\ & T_{W_{\omega, \tau, \xi}}=\operatorname{diag}\left[e\left(\frac{1}{5}\right), e\left(\frac{3}{10}\right), e\left(-\frac{3}{10}\right), e\left(-\frac{1}{5}\right), e\left(\frac{23}{60}\right), e\left(\frac{23}{60}\right), e\left(-\frac{13}{60}\right), e\left(-\frac{13}{60}\right), e\left(-\frac{5}{12}\right),\right. \\ & e\left(\frac{1}{30}\right), e\left(\frac{2}{15}\right), e\left(-\frac{7}{15}\right), e\left(-\frac{11}{30}\right), e\left(\frac{1}{20}\right), e\left(\frac{1}{20}\right), e\left(\frac{9}{20}\right), e\left(\frac{9}{20}\right), e\left(\frac{1}{4}\right), e\left(\frac{1}{30}\right), e\left(\frac{2}{15}\right), \\ & \left.e\left(-\frac{7}{15}\right), e\left(-\frac{11}{30}\right), e\left(\frac{23}{60}\right), e\left(\frac{23}{60}\right), e\left(-\frac{13}{60}\right), e\left(-\frac{13}{60}\right), e\left(-\frac{5}{12}\right)\right] ; \text { and }\end{aligned}$

\begin{equation*}
     S=\\\frac{1}{\lambda_1}\begin{bmatrix} s_{X,X} & \chi_{15}^4s_{X,X}&\chi_{15}^5s_{X,Z}&\chi_{15}^3s_{X,W}\\\chi_{15}^4s_{X,X}& s_{X,X} &\chi_{15}^5s_{X,Z}&-\chi_{15}^3s_{X,W}\\\chi_{15}^5s_{X,Z}^T&\chi_{15}^5s_{X,Z}^T,&\chi_{15}^5s_{Z,Z}&\mathbf 0\\ \chi_{15}^3s_{X,W}^T&-\chi_{15}^3s_{X,W}^T&\mathbf 0&s_{W,W}\end{bmatrix}, 
\end{equation*}

\begin{align*}
    & \lambda_1 = 6\sqrt{10\chi_{15}^4},\quad s_{X_g,X_h} = e\left(\frac{gh}{3}\right), \quad s_{X_g,Z_{k,l}}= e\left(\frac{g(k+l)}{6}\right) ,\quad  s_{Z_{k,l}, Z_{k',l'}} = e\left(\frac{kk'+ll'}{6}\right) + e\left(\frac{kl'+k'l}{6}\right), \\
&\text{where } g, h, k, k', l, l' \in \mathbb{Z} / 6, k\neq l, k'\neq l',
\end{align*}
and the entries for  $s_{X,W}$ and $s_{W,W}$ can be obtained from equations (\ref{Smatrix-center})and (\ref{Smatrix-S44}) using the solutions in Table \ref{table:2}. We provide detailed modular data in the Mathematica notebook \texttt{CenterofJ61.nb} in the \texttt{arxiv} source.
\begin{theorem}\label{thm:MD1}
 Let $\mathcal C$ be the Drinfeld center of a near-group category of type $\mathbb Z/6+6$. Then $\mathcal C\cong \mathcal D\boxtimes \mathcal C(\mathbb Z/3, q)$, where $\mathcal D$ is a spin modular category, and $q$ is its associated quadratic form restricted to $\mathbb Z/3$. Moreover, the Müger centralizer of the fermion $f$ in $\mathcal D$ is super-modular and either itself or one of its Galois conjugates has the same super-modular data as in (\ref{SMDS1}). 
\end{theorem}
\begin{proof}
We first consider the case $J_6^1$. Take the pointed modular category $\mathcal C(\mathbb Z/3, q)$ with $q=e^{2\pi i a^2/3}$, $a\in \mathbb Z/3$.  Notice $\mathcal{C}$ has $\mathcal C(\mathbb Z/3, q)$ as a pointed modular subcategory, generated by the invertible objects $X_0, X_2$ and $X_4$. Also, one observes that the invertible object $X_3$ is a fermion. Thus $\mathcal{C} \cong \mathcal D\boxtimes \mathcal C(\mathbb Z/3, q)$, where $\mathcal{D}$ is a spin modular category with the fermion $f=X_3$. Moreover, $\mathcal D\cong \mathcal{D}_0\oplus\mathcal{D}_1$, where $\DD_0=C_\mathcal{D}(\langle f \rangle)$ is a super-modular category.

Specifically, the spin modular category $\mathcal D$ is generated by the simple objects: $\mathbf 1$, $X_3$, $Y_0$, $Y_3$, $Z_{0,1}$, $Z_{0,3}$, $Z_{0,5}$, $Z_{1,5}$,  $Z_{2,4}$, $W_1,\ldots, W_{9}$. The 10 simple objects in $\mathcal D_0$ are  $\mathbf 1$, $X_3$, $Y_0$, $Y_3$,  $Z_{1,5}$,  $Z_{2,4}$, $W_1,\ldots, W_{4}$.  The resulting super-modular data is 
\begin{equation}
\hat{S} = \frac{1}{\sqrt{30 \chi_{15}^4}}
\begin{bmatrix}
1 & \chi_{15}^4 & \chi_{15}^5 & \chi_{15}^3 & \chi_{15}^3 \\
\chi_{15}^4 & 1 & \chi_{15}^5 & -\chi_{15}^3 & -\chi_{15}^3 \\
\chi_{15}^5 & \chi_{15}^5 & -\chi_{15}^5 & 0 & 0 \\
\chi_{15}^3 & -\chi_{15}^3 & 0 & \frac{1}{2} \chi_5^1 \chi_{15}^3 & -\frac{2 \sqrt{30 \chi_{15}^4}}{\chi_5^5} \\
\chi_{15}^3 & -\chi_{15}^3 & 0 & -\frac{2 \sqrt{30 \chi_{15}^4}}{\chi_5^5} & \frac{1}{2} \chi_5^1 \chi_{15}^3
\end{bmatrix},\quad \hat{T}^2 = \diag[1, 1,  e^{2 i \pi / 3}, e^{4i\pi/5}, e^{-4i\pi/5}],
\end{equation}
which coincides with  (\ref{SMDS1}). 

The other cases in Table \ref{table:1} can be similarly verified. Each of them provides a realization of the super-modular data in \cite[Table 6]{cho2022classification} with the global dimension 472.379.
\end{proof}

\subsection{$G=\mathbb Z/2\times \mathbb Z/4$} Up to equivalence, there are 4 inequivalent near group categories of type $\mathbb Z/2\times \mathbb Z/4+8$ listed in Table \ref{table:3} \cite[Proposition 6]{EvansGannon}.
\begin{table}[h]
\begin{tabular}{|c|c|c|c|c|}
\hline
\#            & $c$                & $\langle\, , \,\rangle$ & $a$          & $j(0,1),  j(0, 2),  j(1,0), j(1,1),  j(1,2)$                                                    \\ \hline
$J_{(2,4)}^1$ & $\zeta_{12}^{5}$  & $1$                     & $(1,\, 1)$   & $-0.992441,1.5708,0.785398,-1.42977,-0.785398$     \\ \hline
$J_{(2,4)}^2$ & $\zeta_{12}^{-5}$ & $-1$                     & $(-1,\, 1)$  & $0.992441,-1.5708,-0.785398,1.42977,0.785398$      \\ \hline
$J_{(2,4)}^3$ & $\zeta_{12}^{-5}$ & $1$                     & $(1,\, -1)$  & $1.42977,-1.5708,0.785398,-1.77784,-0.785398$      \\ \hline
$J_{(2,4)}^4$ & $\zeta_{12}^{5}$  & $-1$                    & $(-1,\, -1)$ & $-1.42977,1.5708,-0.785398,1.77784,0.785398$ \\ \hline
\end{tabular}
\caption{The data for near-group categories of type $\mathbb Z/2\times\mathbb Z/4+8$.}
\label{table:3}
\end{table}
In Table \ref{table:3}, the bicharacter for $\langle\, , \,\rangle$ is of the form $\langle (x_1, y_1), (x_2, y_2) \rangle=\exp \left(2 \pi i x_1 x_2 / 2\right) \exp \left(2 \pi i m y_1 y_2/4\right)$, where $m $ is given in the $\langle\, , \,\rangle$ column.  The functions $a$ are given by $a(x, y) = s_1^x s_2^y \exp \left(- \pi i x^2 / 2\right) \exp \left(- \pi i m y^2/4\right)$, where $s_1, s_2 \in\{\pm 1\}$ are given in the column for $a$. Note $b(x)=\frac{1}{\sqrt{n}} \exp (i j(x))$ if $x\neq 0$, where $j(x)\in [-\pi, \pi]$,  we list the values for  $j(x)$ in Table \ref{table:3} for the 5 values of $x$ needed to determine $b(x)$: the remaining values can be obtained from $b(0)=-\frac{1}{d}$ and $b(-x)=\overline{a(x) b(x)}$.

We consider the near-group $J_{(2,4)}^1$ and solve for the triples $(\omega_j, \tau_j, \xi_j)$ using the equations \eqref{eq:half1} - \eqref{eq:half4}. The solutions are listed in Table \ref{table:4}. The $\omega_j$  equals to $\zeta_{48}^k$ for some $k\in \mathbb Z$, which are recorded in the column of $\omega$. 
Since $\xi_j(g)=\exp(i\theta_{j,g})$, the values of $\theta_{j,g}$ are listed in the column of $\xi$
with the order  $(0, 0),\ldots,(0,3),(1,0),\ldots, (1,3) \in \mathbb Z/2\times\mathbb Z/4$. We also provide the solutions to $(\omega_j, \tau_j, \xi_j)$ and further computation of modular data in the \texttt{arxiv} source \texttt{CenterofJ241.nb}.

\begin{table}[H]
\begin{tabular}{|l|c|c|c|}
\hline
\# & $\omega$ & $\tau$   & $\xi$                                                    \\ \hline
1  & $16$     & $(0, 0)$ & $0,0.4373,0,-2.008,-1.571,3.076,1.571,-1.505$            \\
2  & $16$     & $(0, 0)$ & $0,-2.008,0,0.4373,1.571,-1.505,-1.571,3.076$            \\
3  & $40$     & $(0, 0)$ & $1.571,-1.505,-1.571,3.076,0,-2.008,0,0.4373$            \\
4  & $40$     & $(0, 0)$ & $-1.571,3.076,1.571,-1.505,0,0.4373,0,-2.008$            \\ \hline
5  & $21$     & $(0, 1)$ & $-1.134,1.003,-0.0339,3.045,-0.0339,3.045,-1.134,1.003$  \\
6  & $45$     & $(0, 1)$ & $-0.0339,3.045,-1.134,1.003,-1.134,1.003,-0.0339,3.045$  \\
7  & $13$     & $(0, 1)$ & $0.2202,-1.398,-1.675,-2.644,-2.842,-1.477,2.793,2.312$  \\
8  & $13$     & $(0, 1)$ & $2.793,2.312,-2.842,-1.477,-1.675,-2.644,0.2202,-1.398$  \\
9  & $37$     & $(0, 1)$ & $-1.675,-2.644,0.2202,-1.398,2.793,2.312,-2.842,-1.477$  \\
10 & $37$     & $(0, 1)$ & $-2.842,-1.477,2.793,2.312,0.2202,-1.398,-1.675,-2.644$  \\ \hline
11 & $4$      & $(0, 2)$ & $3.076,1.571,-1.505,-1.571,0.4373,0,-2.008,0$            \\
12 & $4$      & $(0, 2)$ & $-1.505,-1.571,3.076,1.571,-2.008,0,0.4373,0$            \\
13 & $28$     & $(0, 2)$ & $0.4373,0,-2.008,0,3.076,1.571,-1.505,-1.571$            \\
14 & $28$     & $(0, 2)$ & $-2.008,0,0.4373,0,-1.505,-1.571,3.076,1.571 $           \\ \hline
15 & $21$     & $(0, 3)$ & $1.003,-0.0339,3.045,-1.134,3.045,-1.134,1.003,-0.0339$  \\
16 & $45$     & $(0, 3)$ & $3.045,-1.134,1.003,-0.0339,1.003,-0.0339,3.045,-1.134$  \\
17 & $13$     & $(0, 3)$ & $2.312,-2.842,-1.477,2.793,-2.644,0.2202,-1.398,-1.675$  \\
18 & $13$     & $(0, 3)$ & $-1.398,-1.675,-2.644,0.2202,-1.477,2.793,2.312,-2.842$  \\
19 & $37$     & $(0, 3)$ & $-2.644,0.2202,-1.398,-1.675,2.312,-2.842,-1.477,2.793$  \\
20 & $37$     & $(0, 3)$ & $-1.477,2.793,2.312,-2.842,-1.398,-1.675,-2.644,0.2202$  \\ \hline
21 & $22$     & $(1, 0)$ & $-2.241,-1.335,-0.835,0.8973,1.455,3.03,0.04963,-1.022$  \\
22 & $22$     & $(1, 0)$ & $-0.835,0.8973,-2.241,-1.335,0.04963,-1.022,1.455,3.030$ \\
23 & $22$     & $(1, 0)$ & $0.04963,-1.022,1.455,3.03,-0.835,0.8973,-2.241,-1.335$  \\
24 & $22$     & $(1, 0)$ & $1.455,3.03,0.04963,-1.022,-2.241,-1.335,-0.835,0.8973$  \\
25 & $6$      & $(1, 0)$ & $2.235,2.673,2.235,2.673,1.168,-0.8402,1.168,-0.8402$    \\
26 & $6$      & $(1, 0)$ & $1.168,-0.8402,1.168,-0.8402,2.235,2.673,2.235,2.673$    \\ \hline
27 & $15$     & $(1, 1)$ & $0.6315,-3.119,2.979,-2.324,0.6315,-3.119,2.979,-2.324$  \\
28 & $39$     & $(1, 1)$ & $2.979,-2.324,0.6315,-3.119,2.979,-2.324,0.6315,-3.119$  \\
29 & $7$      & $(1, 1)$ & $2.65,1.091,2.736,-0.06957,1.658,0.09927,-0.3231,-3.129$ \\
30 & $7$      & $(1, 1)$ & $1.658,0.09927,-0.3231,-3.129,2.65,1.091,2.736,-0.06957$ \\
31 & $31$     & $(1, 1)$ & $2.736,-0.06957,2.65,1.091,-0.3231,-3.129,1.658,0.09927$ \\
32 & $31$     & $(1, 1)$ & $-0.3231,-3.129,1.658,0.09927,2.736,-0.06957,2.65,1.091$ \\ \hline
33 & $34$     & $(1, 2)$ & $-1.022,1.455,3.03,0.04963,0.8973,-2.241,-1.335,-0.835$  \\
34 & $34$     & $(1, 2)$ & $0.8973,-2.241,-1.335,-0.835,-1.022,1.455,3.03,0.04963$  \\
35 & $34$     & $(1, 2)$ & $3.03,0.04963,-1.022,1.455,-1.335,-0.835,0.8973,-2.241$  \\
36 & $34$     & $(1, 2)$ & $-1.335,-0.835,0.8973,-2.241,3.03,0.04963,-1.022,1.455$  \\
37 & $18$     & $(1, 2)$ & $2.673,2.235,2.673,2.235,-0.8402,1.168,-0.8402,1.168$    \\
38 & $18$     & $(1, 2)$ & $-0.8402,1.168,-0.8402,1.168,2.673,2.235,2.673,2.235$    \\ \hline
39 & $15$     & $(1, 3)$ & $-3.119,2.979,-2.324,0.6315,-3.119,2.979,-2.324,0.6315$  \\
40 & $39$     & $(1, 3)$ & $-2.324,0.6315,-3.119,2.979,-2.324,0.6315,-3.119,2.979$  \\
41 & $7$      & $(1, 3)$ & $1.091,2.736,-0.06957,2.65,0.09927,-0.3231,-3.129,1.658$ \\
42 & $7$      & $(1, 3)$ & $0.09927,-0.3231,-3.129,1.658,1.091,2.736,-0.06957,2.65$ \\
43 & $31$     & $(1, 3)$ & $-3.129,1.658,0.09927,-0.3231,-0.06957,2.65,1.091,2.736$ \\
44 & $31$     & $(1, 3)$ & $-0.06957,2.65,1.091,2.736,-3.129,1.658,0.09927,-0.3231$ \\ \hline
\end{tabular}

\caption{ $(\omega_j, \tau_j, \xi_j)$ for $J_{(2,4)}^1$ in Table \ref{table:3}.}
\label{table:4}
\end{table}

The center of the near-group category $J_{(2,4)}^1$ has rank $88$ with the following simple objects:
\begin{itemize}
    \item[] eight invertible objects $X_{g}$, $g\in \mathbb Z/2\times \mathbb Z/4$;
    \item[] eight $\chi_{24}^5$-dimensional objects $Y_h$, $h\in \mathbb Z/2\times \mathbb Z/4$;
    \item[] twenty-eight $2\chi_{6}^3$-dimensional objects $Z_{k,l}$, $k,l\in \mathbb Z/2\times\mathbb Z/4$, $k\neq l$; and
    \item[] forty-four $\chi_{24}^4$-dimensional objects $W_{\omega_i,\tau_i, \xi_i}$, where $(\omega_i,\tau_i, \xi_i)$ are solutions in Table \ref{table:4}. 
\end{itemize}

The modular data is given as the following:

$\begin{aligned} & T_{X_g}=e\left(\frac{2 g_1^2+g_2^2}{4}\right), \quad T_{Y_h}=e\left(\frac{2 h_1^2+h_2^2}{4}\right), \quad T_{Z_{k, l}}=e\left(\frac{2 k_1 l_1+k_2 l_2}{4}\right), \text { where } g, h, k, l \in \mathbb{Z} / 2 \times \mathbb{Z} / 4, k \neq l.\end{aligned} $

\begin{align*}& T_{W_{\omega, \tau, \xi}}=\operatorname{diag}\left[e\left(\frac{1}{3}\right), e\left(\frac{1}{3}\right), e\left(-\frac{1}{6}\right), e\left(-\frac{1}{6}\right), e\left(\frac{7}{16}\right), e\left(-\frac{1}{16}\right), e\left(\frac{13}{48}\right), e\left(\frac{13}{48}\right), e\left(-\frac{11}{48}\right), e\left(-\frac{11}{48}\right),\right. \\ & e\left(\frac{1}{12}\right), e\left(\frac{1}{12}\right), e\left(-\frac{5}{12}\right) e\left(-\frac{5}{12}\right), e\left(\frac{7}{16}\right), e\left(-\frac{1}{16}\right), e\left(\frac{13}{48}\right), e\left(\frac{13}{48}\right), e\left(-\frac{11}{48}\right), e\left(-\frac{11}{48}\right), e\left(\frac{11}{24}\right), e\left(\frac{11}{24}\right) \\ & e\left(\frac{11}{24}\right), e\left(\frac{11}{24}\right), e\left(\frac{1}{8}\right), e\left(\frac{1}{8}\right), e\left(\frac{5}{16}\right), e\left(-\frac{3}{16}\right), e\left(\frac{7}{48}\right), e\left(\frac{7}{48}\right), e\left(-\frac{17}{48}\right), e\left(-\frac{17}{48}\right), e\left(-\frac{7}{24}\right), e\left(-\frac{7}{24}\right) \\ & \left.e\left(-\frac{7}{24}\right), e\left(-\frac{7}{24}\right), e\left(\frac{3}{8}\right), e\left(\frac{3}{8}\right), e\left(\frac{5}{16}\right), e\left(-\frac{3}{16}\right), e\left(\frac{7}{48}\right), e\left(\frac{7}{48}\right), e\left(-\frac{17}{48}\right), e\left(-\frac{17}{48}\right)\right] ; \text { and }\end{align*}

\begin{align*}
&S=\frac{1}{\lambda_2} \begin{bmatrix} s_{X,X} & \chi_{24}^5s_{X,X}&2\chi_6^3s_{X,Z}&\chi_{24}^4s_{X,W}\\ \chi_{24}^5s_{X,X}& s_{X,X} &2\chi_6^3s_{X,Z}&-\chi_{24}^4s_{X,W}\\2\chi_6^3s_{X,Z}^T&2\chi_6^3s_{X,Z}^T,&2\chi_6^3s_{Z,Z}&\mathbf 0\\ \chi_{24}^4s_{X,W}^T&-\chi_{24}^4s_{X,W}^T&\mathbf 0&s_{W,W}\end{bmatrix}, 
\end{align*}

\begin{align*}
    &\lambda_2 = 16 \sqrt{3\chi_{24}^5},\quad s_{X_g,X_h}= e\left(-\frac{2g_1h_1+g_2h_2}{2}\right) , \quad s_{X_g,Z_{k,l}}= e\left(-\frac{2g_1(k_1+l_1)+g_2(k_2+l_2)}{4}\right),\\ &s_{Z_{k,l}, Z_{k',l'}} = e\left(-\frac{2k_1k_1'+k_2k_2'+2l_1l_1'+l_2l_2'}{4}\right) + e\left(-\frac{2k_1'l_1+k_2'l_2+2k_1l_1'+k_2l_2'}{4}\right),
\end{align*}
where $g, h, k, k^{\prime}, l, l^{\prime} \in \mathbb Z/2\times\mathbb Z/4, k \neq l, k^{\prime} \neq l^{\prime},$

and $s_{X,W}$ and $s_{W,W}$ can be computed from equations \ref{Smatrix-center} and \ref{Smatrix-S44} using the solutions in Table \ref{table:4}.

\begin{theorem}\label{thm:MD2}
    Let $\mathcal{C}$ be the Drinfeld center of a near-group category of type $\mathbb Z/2\times \mathbb Z/4 +8$. Then $\left[\mathcal{C}_{\mathbb \Z/2}\right]_0\cong \mathcal{D}\boxtimes \mathcal{C}(\mathbb Z/2, q)$, where $\mathcal{D}$ is a spin modular category and $q$ is the associated quadratic form restricted to $\mathbb Z/2$. Moreover, the Müger centralizer of the fermion $f$ in $\mathcal D$ is super-modular and either itself or one of its Galois conjugates has the same super-modular data as in (\ref{SMDS2}).  
\end{theorem}

\begin{proof}
First, we examine the case $J_{2,4}^1$ in Table \ref{table:3}. Notice that the invertible simple object $X_{(0,2)}$ has twist 1, thus it is a boson. We can apply the boson condensation described in Section \ref{subsec:anyoncondensation} to obtain a modular category $\left[\mathcal{C}_{\mathbb Z/2}\right]_0$. Using the Verlinde formula with the  $S$ matrix associated to  $J_{2,4}^1$, we can obtain the fusion rules of $\mathcal{C}$. We collect the fusion rules of simple objects tensoring with $b$ and those centralized by $b$ in Table \ref{table:5}. Therefore, $[\mathcal{C}_{\mathbb Z/2}]_0$ is a modular category with the 36 simple objects in Table \ref{table:6}.

\begin{table}[h]
\begin{tabular}{|c|p{15.48cm}|}
\hline
dim &
   \begin{tabular} {|p{5.3cm}|p{8.2cm}|p{1.1cm}|}
$\qquad\qquad\qquad b\otimes -$ & & in $\langle b\rangle^\prime$?\end{tabular}\\
\hline
1&\begin{tabular}{|p{5.3cm}|p{8.2cm}|p{1.1cm}|}
$bX_{(i,j)}=bX_{(i,j+2)}$ & $(i,j)\in\mathbb Z/2\times\mathbb Z/4$ &Yes\end{tabular}\\
\hline
$\chi_{24}^5$&\begin{tabular}{|p{5.3cm}|p{8.2cm}|p{1.1cm}|}
$bY_{(i,j)}=bY_{(i,j+2)}$&$(i,j)\in\mathbb Z/2\times\mathbb Z/4$ &Yes\end{tabular}\\
\hline
$2\chi_6^3$&\begin{tabular}{|p{5.3cm}|p{8.2cm}|p{1.1cm}|}
$bZ_{(k_1,k_1),(l_1,l_2)}=Z_{(k_1,k_2),(l_1,l_2)}$&$(k_1,k_2),(l_1,l_2)=(0,0),(0,2);(0,1),(0,3);$&Yes\\
&$(1,0),(1,2);(1,1),(1,3)$&\\
\hline
$bZ_{(k_1,k_1),(l_1,l_2)}=Z_{(k_1,k_2+2),(l_1,l_2+2)}$&$(k_1,k_2),(l_1,l_2)=(0,0),(1,0);(0,0),(1,2);$&Yes\\
&$(0,1),(1,1);(0,1),(1,3)$&\\
\hline
$bZ_{(k_1,k_2),(l_1,l_2)}=Z_{(k_1,k_2+2),(l_1,l_2+2)}$&$(k_1,k_2),(l_1,l_2)=(0,0),(0,1);(0,0),(1,1);(0,0),(1,3);$&No\\
&($0,1),(1,0);(0,1),(1,2);(1,0),(1,1)$&\\
\hline
$bZ_{(k_1,k_2),(l_1,l_2)}=Z_{(k_1,k_2+1),(l_1,l_2-1)}$&$(k_1,k_2),(l_1,l_2)=(0,0),(0,3);(1,0),(1,3)$&No\end{tabular}\\
\hline
    $\chi_{24}^4$&\begin{tabular}{|p{5.3cm}|p{8.2cm}|p{1.1cm}|}
$bW_i=W_i$&$i=25,26,37,38$&Yes\\
\hline
$bW_i=W_{i+1}$&$i=1,3,11,13,21,23$&Yes\\
\hline
$bW_i=W_{i+1}$&$i=5,15,18,27,39,42$&No\\
\hline
$bW_i=W_{i+2}$&$i=7,8,29,30$&No\\
\hline
$bW_i=W_{i+2}$&$i=33,34$&Yes\\
\hline
$bW_i=W_{i+3}$&$i=17,41$&No\end{tabular}\\
\hline
\end{tabular}
\caption{Fusion rules of tensoring with $b$ and centralization by $b$}
\label{table:5}
\end{table}

\begin{table}[h!]
\begin{tabular}{|c|c|c|c|}
\hline
$\dim$        & Objects      & Twists               & Number Count \\ \hline
1             & \begin{tabular}[c]{@{}c@{}}$F(\mathbf 1)$, $F\left(X_{(0,1)}\right)$,  \\ $F\left(X_{(1,0)}\right)$, $F\left(X_{(1,1)}\right)$\end{tabular}                                                                                                                                                                           & $1, i, -1, -i $                                                                                                                                                                                                                            & $4$          \\ \hline
$\chi_{24}^5$ & \begin{tabular}[c]{@{}c@{}}$F\left(Y_{(0,0)}\right)$, $F\left(Y_{(0,1)}\right)$,  \\ $F\left(Y_{(1,0)}\right)$, $F\left(Y_{(1,3)}\right)$\end{tabular}                                                                                                                                                                & $1, i,  -1, -i $                                                                                                                                                                                                                           & $4$          \\ \hline
$2\chi_6^3$   & \begin{tabular}[c]{@{}c@{}}$F\left(Z_{(0,0),(1,0)}\right)$, $F\left(Z_{(0,0),(1,2)}\right)$, \\ $F\left(Z_{(0,1),(1,1)}\right)$, $F\left(Z_{(0,1),(1,3)}\right)$\end{tabular}                                                                                                                                         & $ 1, 1, i, -i $                                                                                                                                                                                                                           & $4$          \\ \hline
$\chi_6^3$    & \begin{tabular}[c]{@{}c@{}}$\left(Z_{(0,0),(0,2)}\right)_1,\left(Z_{(0,0),(0,2)}\right)_2,$\\ $\left(Z_{(0,1),(0,3)}\right)_1,\left(Z_{(0,1),(0,3)}\right)_2,$\\ $\left(Z_{(1,0),(1,2)}\right)_1,  \left(Z_{(1,0),(1,2)}\right)_2,$\\ $\left(Z_{(1,1),(1,3)}\right)_1$, $\left(Z_{(1,1),(1,3)}\right)_2$\end{tabular} & $ 1,1,-i,-i,-1,-1,i,i$                                                                                                                                                                                                                       & $8$          \\ \hline
$\chi_{24}^4$ & \begin{tabular}[c]{@{}c@{}}$F\left(W_1\right)$,  $F\left(W_3\right)$,  $F\left(W_{11}\right)$, \\ $F\left(W_{13}\right)$, $F\left(W_{21}\right)$,  $F\left(W_{23}\right)$,  \\ $F\left(W_{33}\right)$, $F\left(W_{34}\right)$\end{tabular}                                                                            & \begin{tabular}[c]{@{}c@{}}$e^{\frac{2 i \pi }{3}},e^{-\frac{i \pi}{3} },e^{\frac{i \pi }{6}},e^{-\frac{5 i \pi}{6}},$\\ $e^{\frac{11 i \pi}{12}}, e^{\frac{11 i \pi}{12}}, e^{-\frac{7 i \pi}{12}}, e^{-\frac{7 i\pi}{12} }$\end{tabular} & $8$          \\ \hline
$\chi_6^2$    & \begin{tabular}[c]{@{}c@{}}$\left(W_{25}\right)_1,\left(W_{25}\right)_2,\left(W_{26}\right)_1,$\\ $\left(W_{26}\right)_2,\left(W_{37}\right)_1,\left(W_{37}\right)_2,$\\ $\left(W_{38}\right)_1$,  $\left(W_{38}\right)_2$\end{tabular}                                                                               & \begin{tabular}[c]{@{}c@{}}$e^{\frac{i \pi }{4}},e^{\frac{i \pi }{4}},e^{\frac{i \pi }{4}},e^{\frac{i \pi }{4}},$\\ $e^{\frac{3 i \pi }{4}},e^{\frac{3 i \pi }{4}},e^{\frac{3 i \pi }{4}},e^{\frac{3 i \pi }{4}} $\end{tabular}            & $8$          \\ \hline
\end{tabular}

\caption{ Simple objects in $\left[\mathcal{C}_{\mathbb Z/2}\right]_0$}
\label{table:6}
\end{table}

 As $\theta_{F(X_{(0,1)})}=i $, $\left[\mathcal{C}_{\mathbb Z/2}\right]_0$ has a modular subcategory generated by the semion  $F(X_{(0,1)})=F(X_{(0,3)})$, which is equivalent to $\mathcal{C}(\mathbb Z/2, q)$, where $q= i^{x^2}$, $x\in \mathbb Z/2$. Thus $[\mathcal{C}_{\mathbb Z/2}]_0\cong\mathcal{D} \boxtimes \mathcal{C}(\mathbb{Z} / 2, q)$, where $\mathcal{D}$ contains the simple object $F(X_{(1,0)})=F(X_{(1,2)})$, which is a fermion since $\theta_{F(X_{1,0})}=-1$. We list the simple objects in the spin modular category $\DD$ in Table \ref{table:7}.
 
\begin{table}[h]
\begin{tabular}{|c|c|c|}
\hline
$\dim$         & Objects                                                                                                                       & Twists                                                                                      \\ \hline
$1$            & $F(\mathbf 1)$, $F({X_{(1,0)}})$                                                                                             & $1, -1$                                                                                     \\ \hline
$\chi_{24}^5 $ & $F(Y_{(0,0)}), F(Y_{(1,0)})$                                                                                                    & $1,  -1 $                                                                                   \\ \hline
$2\chi_6^3$    & $F(Z_{(0,0),(1,0)}), F(Z_{(0,0),(1,2)})$                                                                               & $1,\, 1$                                                                                      \\ \hline
$ \chi_6^3$    & \begin{tabular}[c]{@{}c@{}}$\left(Z_{(0,1),(0,3)}\right)_1, \left(Z_{(0,1),(0,3)}\right)_2,$\\ $\left(Z_{(1,1),(1,3)}\right)_1, \left(Z_{(1,1),(1,3)}\right)_2 $\end{tabular} & $-i, -i, i,i $                                                                              \\ \hline
$\chi_{24}^4$  & $ F(W_1), F(W_3), F(W_{21}), F(W_{23})$                                                                                       & $e^{\frac{2\pi i}{3}}, e^{-\frac{\pi i}{3}}, e^{\frac{11\pi i}{12}},e^{\frac{11\pi i}{12}}$ \\ \hline
$\chi_{6}^2$   & $\left(W_{25}\right)_1,\left(W_{25}\right)_2,\left(W_{26}\right)_1, \left(W_{26}\right)_2$                                    & $e^{\frac{i \pi }{4}},e^{\frac{i \pi }{4}},e^{\frac{i \pi }{4}},e^{\frac{i \pi }{4}}$       \\ \hline
\end{tabular}
\caption{ Simple objects in $\DD$}
\label{table:7}
\end{table}

The centralizer of the fermion $\DD_0$ is generated by the following 10 simple objects: $F(\mathbf 1)$, $F({X_{(1,0)}})$ , $F(Y_{(0,0)}), F(Y_{(1,0)})$, $\left(Z_{(0,1),(0,3)}\right)_1, \left(Z_{(0,1),(0,3)}\right)_2,$ $\left(Z_{(1,1),(1,3)}\right)_1, \left(Z_{(1,1),(1,3)}\right)_2 $, $F(W_1), $ and $F(W_3)$.

Using Theorem \ref{thm:s-matrixcondense} we find that the $\hat{S}$-matrix corresponding to $\DD_0$ is:
$$\frac{1}{\hat{\lambda}_2}\left[\begin{array}{ccccc}
1 & 5+2 \sqrt{6} & 3+\sqrt{6} & 3+\sqrt{6} & 4+2 \sqrt{6} 
\\
 5+2 \sqrt{6} & 1 & t  & r  & -4-2 \sqrt{6} 
\\
 3+\sqrt{6} & t  & \mathit{u}_1  & \mathit{u}_2  & x  
\\
 3+\sqrt{6} & r  & \mathit{u}_2  & \mathit{u}_3  & y  
\\
 4+2 \sqrt{6} & -4-2 \sqrt{6} & x  & y  & 4+2 \sqrt{6} 
\end{array}\right]
$$ where $\hat{\lambda}_2=2 \sqrt{6 \chi_{24}^5}$, $t+r=2(3+\sqrt{6})$ and $x+y=0$.  Orthogonality implies that $x=y=0$ and $t=r=3+\sqrt{6}$, and $\mathit{u}_1 = \mathit{u}_3,\; u_2 = -2\sqrt{6} - u_3 - 6$. 

Thus we have:

$$\frac{1}{\hat{\lambda}_2}\left[\begin{array}{ccccc}
1 & 5+2 \sqrt{6} & 3+\sqrt{6} & 3+\sqrt{6} & 4+2 \sqrt{6} 
\\
 5+2 \sqrt{6} & 1 & 3+\sqrt{6} & 3+\sqrt{6} & -4-2 \sqrt{6} 
\\
 3+\sqrt{6} & 3+\sqrt{6} & \mathit{u}_3  & -2 \sqrt{6}-\mathit{u}_3 -6 & 0 
\\
 3+\sqrt{6} & 3+\sqrt{6} & -2 \sqrt{6}-\mathit{u}_3 -6 & \mathit{u}_3  & 0 
\\
 4+2 \sqrt{6} & -4-2 \sqrt{6} & 0 & 0 & 4+2 \sqrt{6} 
\end{array}\right]
$$
Since we don't \emph{a priori} know if the objects of dimension $3+\sqrt{6}$ are a dual pair or self-dual, we must analyze both possibilities.

If they are self-dual orthogonality implies: $-2 \mathit{u}_3^{2}+\left(-4 \sqrt{6}-12\right) \mathit{u}_3 +30+12 \sqrt{6}=0$ so that $\mathit{u}_3=-\sqrt{6} - 3 \pm( 3\sqrt{2} -2\sqrt{3})$.  Using the Verlinde formula for the naive fusion rules \cite[Proposition 2.7]{bruillard2020classification} we compute all the $\hat{N}_{i,j}^k$ and find that these choices of $\mathit{u}_3$ yield negative fusion coefficients, a contradiction.

If they are a dual pair $\mathit{u}_3$ satisfies: $2 \mathit{u}_3^{2}+\left(4 \sqrt{6}+12\right) \mathit{u}_3 +90+36 \sqrt{6}=0$ which yields $\mathit{u}_3=3i \sqrt{2}+2 \,i \sqrt{3}-\sqrt{6}-3$ or $\mathit{u}_3=
-3 \,i \sqrt{2}-2 \,i \sqrt{3}-\sqrt{6}-3$.  These are complex conjugates of each other, so we get:

$$\frac{1}{\hat{\lambda}_2}\left[\begin{array}{ccccc}
1 & 5+2 \sqrt{6} & 3+\sqrt{6} & 3+\sqrt{6} & 4+2 \sqrt{6} 
\\
 5+2 \sqrt{6} & 1 & 3+\sqrt{6} & 3+\sqrt{6} & -4-2 \sqrt{6} 
\\
 3+\sqrt{6} & 3+\sqrt{6} & \pm i(3 \sqrt{2}+2  \sqrt{3})-\sqrt{6}-3 & \mp i(3 \sqrt{2}+2  \sqrt{3})-\sqrt{6}-3 & 0 
\\
 3+\sqrt{6} & 3+\sqrt{6} & \mp i(3 \sqrt{2}+2  \sqrt{3})-\sqrt{6}-3 & \pm i(3 \sqrt{2}+2  \sqrt{3})-\sqrt{6}-3 & 0 
\\
 4+2 \sqrt{6} & -4-2 \sqrt{6} & 0 & 0 & 4+2 \sqrt{6} 
\end{array}\right]
$$ where the 4 signs are determined by any single sign, which we now calculate.
 We first note that without loss of generality (see \cite{bruillard2020classification}) the twists are: $[1, 1,i,i,e^{-\pi i/3}]$, with corresponding objects $\one,\beta,\gamma_1,\gamma_2$ and $\eta$ of dimension $1,\chi^5_{24},\chi_6^3,\chi_6^3$ and $\chi_{24}^4$.  We denote by $f$ the fermion, so that the objects in $\DD_0$ can be labeled by the above and $f,f\beta,f\gamma_1,f\gamma_2$ and $f\eta$.  We aim to calculate $\hat{\lambda}_2\hS_{\gamma_1,\gamma_1}=\pm i(3 \sqrt{2}+2  \sqrt{3})-\sqrt{6}-3$.  Since $\gamma_1^*=\gamma_2$ the balancing equation gives us:

 $$-\hS_{\gamma_1,\gamma_1}=(\theta_{\gamma_1})^2\hS_{\gamma_1,\gamma_1}=\frac{1}{\sqrt{2}\hat{\lambda}_2}\sum_\psi N_{\gamma_2,\gamma_1}^\psi \theta_\psi d_\psi$$ where the sum is over the 10 simple objects in $\DD_0$.  The Verlinde formula for $\hS$ is invariant under complex conjugation so we obtain the (naive) fusion rules:

 $N_{\gamma_2,\gamma_1}^\one=1, N_{\gamma_2,\gamma_1}^\beta+N_{\gamma_2,\gamma_1}^{f\beta}=2$ and $N_{\gamma_2,\gamma_1}^\eta+N_{\gamma_2,\gamma_1}^{f\eta}=1$.  
 One can now simply try each of the two possibilities for $\hS_{\gamma_1,\gamma_1}$ with the fusion rules choices and check for consistency.  The only possibility is that $N_{\gamma_2,\gamma_1}^\beta=N_{\gamma_2,\gamma_1}^{f\beta}=1$ and $N_{\gamma_2,\gamma_1}^\eta=1$, so that $\hat{\lambda}_2\hS_{\gamma_1,\gamma_1}=i(3 \sqrt{2}+2  \sqrt{3})-\sqrt{6}-3$.  Noting that $\hT^2=\diag[1, 1,-1,-1,e^{-2\pi i/3}]$ we see that we have the complex conjugate of the data in (\ref{SMDS2}), which completes the proof.

\end{proof}

\section{Conclusions and Discussion}

We have now realized two previously unknown super-modular data using near-group categories.  We hope to use these techniques to further find modular and super-modular categories, aiding in continuing the classification of these categories.  Although there has been much work on understanding the $S$ and $T$ matrices of centers of near-group categories \cite{grossman2020infinite,EvansGannon,Budthesis}, some of these matrices are only conjectural.  We hope to continue to expand the list of categories for which these conjectures have been verified, along with their modular/super-modular data.

\subsection{Additional Examples}

Below in Table \ref{table:8} we have summarized some cases where a familiar category appears as a modular factor of a near-group center.  As we do not provide proofs the reader may take these as speculations. Moreover, we do not specify a category of the given type in Table \ref{table:8} so there is ambiguity in any case. The authors welcome comments and references for verifications/original sources for these statements.  

\begin{table}[h!]
\begin{tabular}{|l|l|l|}
\hline Type of $\FF$ & Conj. form of $\mathcal{Z}(\FF)$ & Notes \\
\hline\hline $\Z/1+1$&$\operatorname{Fib}\boxtimes \operatorname{Fib}^{\operatorname{rev}} $& $\FF=\operatorname{Fib}$ \\
\hline $\Z/2+2 $& $\mathrm{SU}(2)_{10}$ & cf. \cite{BRW}\\
\hline $\Z/3+3$& $\mathrm{G}(2)_3 \boxtimes \CC(\Z/3,Q)$ & $\rank(\mathrm{G}(2)_3)=6$ cf. \cite{NRWW} \\
\hline  $\Z/2\times \Z/2+4$  &$\left(\left[\left[\mathrm{SU}(2)_6^{\boxtimes2}\right]_{\Z/2}\right]_0\right)_{\Z/2}^{\times,\Z/2}$ & cf. \cite{bruillard2020classification}\\
\hline $\Z/4+4$  & $[\PSU(3)_5\boxtimes\CC(\Z/2,Q)]_{\Z/2}^{\times,\Z/2}$ & $\rank(\PSU(3)_5)=7$  \\
\hline $\Z/5+5 $& $\mathcal{B} \boxtimes \CC(\Z/5,Q)$ & $\rank(\mathcal{B})=8$\\
\hline $\Z/6+6 $ &$\DD\boxtimes\CC(\Z/3,Q) $& $\DD$ in Theorem \ref{thm:MD1}\\
\hline $\Z/7+7 $& $\mathcal{B} \boxtimes \CC(\Z/7,Q)$ & $\rank(\mathcal{B})=10$ \\
\hline $\Z/8+8 $ &$[G(2)_4\boxtimes \CC(\Z/4,Q)]_{\Z/2}^{\times,\Z/2}$&  $\rank(G(2)_4)=9$ \\
\hline$\Z/2\times\Z/4+8 $ &$[\DD\boxtimes \CC(\Z/2,Q)]_{\Z/2}^{\times,\Z/2}$ & $\DD$ in Theorem \ref{thm:MD2} \\
\hline
\end{tabular}
\caption{Familiar categories conjecturally related to centers of $G+n$ near-group categories.}
\label{table:8}
\end{table}

A few comments on the notation of Table \ref{table:8}:
\begin{itemize}
\item We denote a pointed modular category by $\CC(A,Q)$ where $Q$ is an unspecified non-degenerate quadratic form.
\item A $G$-gauging $\DD$ of $\BB$ \cite{CGPW} is denoted $\BB_G^{\times, G}$, which is a convenient way of saying that the category $[\DD_G]_0\cong\BB$.  
    \item The 4 rows corresponding to near-group categories associated with groups of order $4$ and $8$ can be understood as follows. The category $\mathcal{Z}(\FF)$ contains a boson, and after boson condensation one obtains the category $\BB$ where the second column has $\BB_{\Z/2}^{\times, \Z/2}$.
\end{itemize}


\bibliographystyle{plain}
\bibliography{Reference.bib}

\end{document}